\documentclass[a4paper, 12pt]{conm-p-l}

\usepackage{amsmath, amsthm, amscd, amsfonts, amssymb, graphicx, color}
\usepackage[bookmarksnumbered, colorlinks, plainpages]{hyperref}
\hypersetup{colorlinks=true,linkcolor=red, anchorcolor=green, citecolor=cyan, urlcolor=red, filecolor=magenta, pdftoolbar=true}
\usepackage{mathrsfs}

\usepackage[OT2,T1]{fontenc}
\DeclareSymbolFont{cyrletters}{OT2}{wncyr}{m}{n}
\DeclareMathSymbol{\Sha}{\mathalpha}{cyrletters}{"58}

\newtheorem{theorem}{Theorem}[section]
\newtheorem{lemma}[theorem]{Lemma}

\newtheorem{corollary}[theorem]{Corollary}

\theoremstyle{definition}
\newtheorem{definition}[theorem]{Definition}
\newtheorem{example}[theorem]{Example}

\theoremstyle{remark}
\newtheorem{remark}[theorem]{Remark}

\numberwithin{equation}{section}

\begin{document}

\title[Shafarevich-Tate groups]{Shafarevich-Tate groups of abelian varieties}

\author[Nikolaev]{{Igor ~V.~Nikolaev}}
\address{Department of Mathematics and Computer Science, 
St.~John's University, 8000 Utopia Parkway,  
New York,  New York 11439}
\email{igor.v.nikolaev@gmail.com}


\subjclass[2020]{Primary 11G10; Secondary 46L85}
\date{January 1, 1994 and, in revised form, June 22, 1994.}


\keywords{Abelian varieties, non-commutative tori}

\begin{abstract}
The Shafarevich-Tate group $\Sha (\mathscr{A})$ measures the failure of the 
Hasse principle for an abelian variety $\mathscr{A}$. 
Using a correspondence between the abelian varieties and the higher dimensional non-commutative
tori,  we prove that $\Sha (\mathscr{A})\cong Cl~(\Lambda)\oplus Cl~(\Lambda)$
or  $\Sha (\mathscr{A})\cong \left(\mathbf{Z}/2^k\mathbf{Z}\right) 
\oplus Cl_{~\mathbf{odd}}~(\Lambda)\oplus  Cl_{~\mathbf{odd}}~(\Lambda)$, 
where  $Cl~(\Lambda)$ is the ideal class group of a ring  $\Lambda$ 
 associated to the K-theory of the non-commutative tori  and $2^k $ divides the order of $Cl~(\Lambda)$. 
 The case of elliptic curves with complex multiplication is considered in detail. 
\end{abstract}

\maketitle

\section{Introduction}
The study of Diophantine equations is the oldest part of mathematics. 
If such an equation has an integer solution, then using the reduction
modulo any  prime $p$, one gets a solution of the 
 equation lying in the finite field $\mathbf{F}_p$ and  a solution 
in the field of real numbers  $\mathbf{R}$. The equation is said to satisfy
the Hasse principle, if the converse  is true.  
For instance,  the quadratic equations  satisfy the Hasse principle,
while  the equation $x^4-17=2y^2$ has a solution over $\mathbf{R}$ 
and  each  $\mathbf{F}_p$,   but no rational solutions.
Measuring the failure of  the Hasse principle is a difficult open 
problem of number theory.

Let $\mathscr{A}_K$ be an abelian variety over the number field $K$
which we assume to be simple, i.e. the $\mathscr{A}_K$ has no proper sub-abelian varieties 
over $K$. 
Denote by $K_v$ the completion of $K$ at the (finite or infinite) place $v$. 
Consider the Weil-Ch\^atelet group $WC(\mathscr{A}_K)$ of the abelian 
variety $\mathscr{A}_K$ and the group homomorphism:
\begin{equation}\label{eq1.1}
WC(\mathscr{A}_K)\rightarrow \prod_v WC(\mathscr{A}_{K_v}). 
\end{equation}
The Shafarevich-Tate group $\Sha (\mathscr{A}_K)$ of 
$\mathscr{A}_K$ is defined as the kernel of homomorphism (\ref{eq1.1}). 
The  variety $\mathscr{A}_K$ satisfies the Hasse principle, 
if and only if,  the group $\Sha (\mathscr{A}_K)$ is trivial. 
Little is known about the $\Sha (\mathscr{A}_K)$ in general. 
The existing methods include  an evaluation of the analytic order of
$\Sha (\mathscr{A}_K)$ based on the second part of the Birch and Swinnerton-Dyer Conjecture 
and an exact calculation of  the  $p$-part of $\Sha (\mathscr{A}_K)$,
see e.g. [Rubin 1989] \cite{Rub1}.

The aim of our note is calculation of the group $\Sha (\mathscr{A}_K)$
based on a functor  $F$ between the $n$-dimensional abelian varieties $\mathscr{A}_K$
and the $2n$-dimensional non-commutative tori  $\mathcal{A}_{\Theta}$,  i.e. the $C^*$-algebras
  generated by the unitary operators $U_1,\dots, U_{2n}$ satisfying the commutation 
relations $\{U_jU_i=e^{2\pi i\theta_{ij}} U_iU_j ~|~ 1\le i,j\le 2n\}$
described  by a skew-symmetric matrix
\begin{equation}\label{eq1.2} 
\Theta=\left(
\begin{matrix} 0              & \theta_{12}  & \dots & \theta_{1, 2n}\cr
             -\theta_{12} & 0  & \dots & \theta_{2, 2n}\cr
              \vdots         & \vdots         & \ddots   &\vdots\cr
             -\theta_{1, 2n} & -\theta_{2, 2n} & \dots & 0 
             \end{matrix}
              \right)\in GL_{2n}(\mathbf{R}). 
\end{equation}
The functor maps  isomorphic abelian varieties $\mathscr{A}$ and $\mathscr{A}'$
to the Morita equivalent  algebras  $\mathcal{A}_{\Theta}=F(\mathscr{A})$
and $\mathcal{A}_{\Theta}'=F(\mathscr{A}')$, see e.g. \cite[Section 1.3]{N}.  
Restricting  $F$ to the simple abelian varieties $\mathscr{A}_K$,  one gets  the non-commutative
tori $\mathcal{A}_{\Theta (k)}=F(\mathscr{A}_K)$, where $\Theta (k)$ is the matrix (\ref{eq1.2})  
defined over  a number field  $k\subset \mathbf{R}$.

\smallskip
Roughly speaking, the idea is this. To calculate the $\Sha (\mathscr{A}_K)$,
we recast   (\ref{eq1.1}) in terms of  the K-theory
of  algebra  $\mathcal{A}_{\Theta (k)}=F(\mathscr{A}_K)$. 
Recall that an isomorphism class of $\mathcal{A}_{\Theta (k)}$ is defined 
by the triple $\left(K_0(\mathcal{A}_{\Theta (k)}), K_0^+(\mathcal{A}_{\Theta (k)}),
\Sigma(\mathcal{A}_{\Theta (k)}\right)$ consisting of the $K_0$-group, the positive cone $K_0^+$ and 
the scale $\Sigma$ of   the algebra $\mathcal{A}_{\Theta (k)}$ [Blackadar 1986] \cite[Section 6]{B},
see also Section 2.3.   
It is proved,  that  $\Sigma(\mathcal{A}_{\Theta (k)})$ is  a torsion group, such that  
$WC(\mathscr{A}_K)\cong\Sigma(\mathcal{A}_{\Theta (k)})$ (lemma \ref{lm3.1}). 
The RHS  of (\ref{eq1.1}) corresponds to the crossed product $C^*$-algebra
$\mathcal{A}_{\Theta (k)}\rtimes_{L_v}\mathbf{Z}$,  where the notation is explained in Section 2.4.2.  
It is proved,  that $\prod_v WC(\mathscr{A}_{K_v})\cong \prod_v K_0\left(\mathcal{A}_{\Theta (k)}\rtimes_{L_v}\mathbf{Z}\right)$
(corollary \ref{cor3.2}).  Thus (\ref{eq1.1}) can
be written in the form:
\begin{equation}\label{eq1.3}
\Sigma(\mathcal{A}_{\Theta (k)})\to  \prod_v K_0\left(\mathcal{A}_{\Theta (k)}\rtimes_{L_v}\mathbf{Z}\right).
\end{equation}

\smallskip
Both sides of (\ref{eq1.3}) are functions of a single positive matrix $B\in GL(2n, \mathbf{Z})$, see definition \ref{dfn2.1}.   
However,  the LHS of  (\ref{eq1.3})  depends on the similarity class of 
$B$, while the RHS of (\ref{eq1.3})  depends   on the characteristic polynomial  of $B$. 
This observation is critical, since it puts the  elements of  $\Sha (\mathscr{A}_K)$
into a one-to-one correspondence with the similarity classes of matrices having the same
characteristic polynomial.  The latter is an old problem of linear algebra and 
it is  known,  that the number of  such classes  is finite.  
They  correspond to  the ideal classes $Cl~(\Lambda)$ of an order   $\Lambda$ in  the field 
$k=\mathbf{Q}(\lambda_B)$, where  $\lambda_B$  is  the Perron-Frobenius eigenvalue of matrix $B$ [Latimer \& MacDuffee 1933] \cite{LatMac1}.  
Let  $Cl~(\Lambda)\cong \left(\mathbf{Z}/2^k\mathbf{Z} \right) \oplus Cl_{~\mathbf{odd}}(\Lambda)$ for $k\ge 0$. 
Using the Atiyah pairing between the K-theory and   the  K-homology, one gets   the following result.  
\begin{theorem}\label{thm1.1}
The Shafarevich-Tate group of an abelian variety $\mathscr{A}_K$
is a finite group given by the formulas:
\begin{equation}
\Sha (\mathscr{A}_K)\cong
\begin{cases}
Cl~(\Lambda)\oplus Cl~(\Lambda), & \mbox{if $k$ is even,}  \cr
\left(\mathbf{Z}/2^k\mathbf{Z}\right) \oplus Cl_{~\mathbf{odd}}~(\Lambda)\oplus  Cl_{~\mathbf{odd}}~(\Lambda),  & \mbox{if $k$ is odd.} 
\end{cases}
\end{equation}
\end{theorem}

\medskip
The article is organized as follows. Preliminary facts can be found 
in Section 2.  The proof of theorem \ref{thm1.1} is  given in Section 3.  
As an illustration, we consider the abelian varieties with complex multiplication 
in Section 4.

\section{Preliminaries}
We briefly review the abelian varieties,  the Shafarevich-Tate groups  and the 
$n$-dimensional non-commutative tori.  For a detailed exposition,  we refer the reader to 
[Mumford 1983] \cite[Chapter 2, \S 4]{M},     [Cassels 1966] \cite{Cas1},
[Blackadar 1986] \cite[Section 6]{B}
and \cite[Section 3.4.1]{N}, respectively.

\subsection{Abelian varieties}
Let $\mathbf{C}^n$ be the $n$-dimensional complex space and $\Lambda$ be
a lattice in the underlying $2n$-dimensional real space $\mathbf{R}^{2n}$. 
The complex torus $\mathbf{C}^n/\Lambda$ is said to be an $n$-dimensional 
(principally polarized) abelian variety $\mathscr{A}$, if it admits an embedding 
into a projective space. In other words, the set $\mathscr{A}$ is both an additive 
abelian group and a complex projective variety. 
Recall that  the Siegel upper half-space $\mathbb{H}_n$ consists of the  symmetric 
$n\times n$ matrices  with complex entries $\tau_i$, such that: 
$\mathbb{H}_n:=\left\{\tau=(\tau_i)\in \mathbf{C}^{{n(n+1)\over 2}}~|~ \Im (\tau_i)>0\right\}$.
 The points of $\mathbb{H}_n$ parametrize 
  the  $n$-dimensional abelian varieties., i.e.  $\mathscr{A}\cong \mathscr{A}_{\tau}$
  for some $\tau\in \mathbb{H}_n$.
Denote by $Sp ~(2n, \mathbf{R})$ the $2n$-dimensional symplectic group,
i.e. a subgroup of the linear group defined by the equation: 
\begin{equation}\label{eq2.1}
\left(
\begin{matrix}
A & B\cr C & D
\end{matrix}
\right)^T
\left(
\begin{matrix}
0 & I\cr -I & 0
\end{matrix}
\right)
\left(
\begin{matrix}
A & B\cr C & D
\end{matrix}
\right)
=
\left(
\begin{matrix}
0 & I\cr -I & 0
\end{matrix}
\right),
\end{equation}
where  $A,B,C, D$ are  the  $n\times n$ matrices with real entries and $I$ is the identity matrix. 
It is well known, that  the abelian varieties $\mathscr{A}_{\tau}$ and $\mathscr{A}_{\tau'}$
are isomorphic, whenever
\begin{equation}\label{eq2.2}
\tau'={A\tau+B\over C\tau+D},  \quad\hbox{where}\quad  
\left(
\begin{matrix}
A & B\cr C & D
\end{matrix}
\right)\in Sp~(2n,\mathbf{Z}). 
\end{equation}

\subsection{Shafarevich-Tate group}
Let $K\subset\mathbf{C}$ be a number field and let $\mathscr{A}_K$ be an 
abelian variety over $K$.   
The principal homogeneous space of $\mathscr{A}_K$ is a variety $\mathscr{C}$
over $K$ with a map $\mu: \mathscr{C}\times \mathscr{A}_K\to\mathscr{C}$
satisfying (i) $\mu(x,0)=x$, (ii) $\mu(x,a+b)=\mu(\mu(x,a),b)$ for all $x\in\mathscr{C}$
and $a,b\in\mathscr{A}_K$ and (iii) for all $x\in\mathscr{C}$ the map $a\mapsto\mu(x,a)$ 
is a $K$-isomorphism between $\mathscr{A}_K$ and $\mathscr{C}$.  
The homogeneous spaces $(\mathscr{C},\mu)$ and $(\mathscr{C}',\mu')$ 
are equivalent, if there is a $K$-isomorphism $\mathscr{C}\to\mathscr{C}'$ 
compatible with the maps $\mu$ and $\mu'$. 
If $\mathscr{C}$ has $K$-points, then the equivalence class  $(\mathscr{C},\mu)$
is said to be trivial. The non-trivial classes  $(\mathscr{C},\mu)$ correspond 
to the varieties $\mathscr{C}$ without $K$-points, but with the $L$-points,
where $L$ is a Galois extension of $K$.  The Weil-Ch\^atelet group 
$WC(\mathscr{A}_K)$ is an additive abelian group of the equivalence 
classes   $(\mathscr{C},\mu)$, where  the trivial class corresponds to the 
zero element of the group,  see  [Cassels 1966] \cite{Cas1} for the details.  
The $WC(\mathscr{A}_K)$  is a torsion group, i.e. each element of  the group has finite order. 
It is not hard to see, that if $K\subset K_v$ is an extension of $K$  by completion at the place $v$,
then there exists a natural group homomorphism:
\begin{equation}\label{eq2.3}
WC(\mathscr{A}_K)\to WC(\mathscr{A}_{K_v}). 
\end{equation}
The kernel of (\ref{eq2.3}) consists of those homogeneous spaces
 having  the $K_v$-points, but no $K$-points. 
  The Shafarevich-Tate group $\Sha (\mathscr{A}_K)$ of 
$\mathscr{A}_K$ is defined as the kernel of homomorphism
$WC(\mathscr{A}_K)\rightarrow \prod_v WC(\mathscr{A}_{K_v})$. 
Thus the abelian variety $\mathscr{A}_K$ satisfies the Hasse principle,  
if and only if,  the group $\Sha (\mathscr{A}_K)$ is trivial.

\subsection{K-theory of $C^*$-algebras}
The  $C^*$-algebra $\mathcal{A}$ is an algebra over $\mathbf{C}$ with a norm
$a\mapsto ||a||$ and an involution $a\mapsto a^*$ such that
it is complete with respect to the norm and $||ab||\le ||a||~ ||b||$
and $||a^*a||=||a||^2$ for all $a,b\in \mathcal{A}$.
Any commutative $C^*$-algebra is  isomorphic
to the algebra $C_0(X)$ of continuous complex-valued
functions on some locally compact Hausdorff space $X$. 
Any other  algebra $\mathcal{A}$ can be thought of as  a non-commutative  
topological space.   By $M_{\infty}(\mathcal{A})$ 
one understands the algebraic direct limit of the $C^*$-algebras 
$M_n(\mathcal{A})$ under the embeddings $a\mapsto ~\mathbf{diag} (a,0)$. 
The direct limit $M_{\infty}(\mathcal{A})$  can be thought of as the $C^*$-algebra 
of infinite-dimensional matrices whose entries are all zero except for a finite number of the
non-zero entries taken from the $C^*$-algebra $\mathcal{A}$.
Two projections $p,q\in M_{\infty}(\mathcal{A})$ are equivalent, if there exists 
an element $v\in M_{\infty}(\mathcal{A})$,  such that $p=v^*v$ and $q=vv^*$. 
The equivalence class of projection $p$ is denoted by $[p]$.   
We write $V(\mathcal{A})$ to denote all equivalence classes of 
projections in the $C^*$-algebra $M_{\infty}(\mathcal{A})$, i.e.
$V(\mathcal{A}):=\{[p] ~:~ p=p^*=p^2\in M_{\infty}(\mathcal{A})\}$. 
The set $V(\mathcal{A})$ has the natural structure of an abelian 
semi-group with the addition operation defined by the formula 
$[p]+[q]:=\mathbf{diag}(p,q)=[p'\oplus q']$, where $p'\sim p, ~q'\sim q$ 
and $p'\perp q'$.  The identity of the semi-group $V(\mathcal{A})$ 
is given by $[0]$, where $0$ is the zero projection. 
By the $K_0$-group $K_0(\mathcal{A})$ of the unital $C^*$-algebra $\mathcal{A}$
one understands the Grothendieck group of the abelian semi-group
$V(\mathcal{A})$, i.e. a completion of $V(\mathcal{A})$ by the formal elements
$[p]-[q]$.  The image of $V(\mathcal{A})$ in  $K_0(\mathcal{A})$ 
is a positive cone $K_0^+(\mathcal{A})$ defining  the order structure $\le$  on the  
abelian group  $K_0(\mathcal{A})$. The pair   $\left(K_0(\mathcal{A}),  K_0^+(\mathcal{A})\right)$
is known as a dimension group of the $C^*$-algebra $\mathcal{A}$. 
The scale $\Sigma(\mathcal{A})$ is the image in $K_0^+(\mathcal{A})$
of the equivalence classes of projections in the $C^*$-algebra $\mathcal{A}$. 
The $\Sigma(\mathcal{A})$ is a generating, hereditary and directed subset 
of  $K_0^+(\mathcal{A})$, i.e. (i) for each $a\in K_0^+(\mathcal{A})$ 
there exist $a_1,\dots, a_r\in\Sigma(\mathcal{A})$ such that 
$a=a_1+\dots+a_r$; (ii) if $0\le a\le b\in \Sigma(\mathcal{A})$, then $a\in\Sigma(\mathcal{A})$
and (iii) given $a,b\in\Sigma(\mathcal{A})$ there exists $c\in\Sigma(\mathcal{A})$,
such that $a,b\le c$.   Each  scale  can always be written as 
$\Sigma(\mathcal{A})=\{a\in K_0^+(\mathcal{A}) ~|~0\le a\le u\}$,
where $u$ is an  order unit of  $K_0^+(\mathcal{A})$.  
The pair  $\left(K_0(\mathcal{A}),  K_0^+(\mathcal{A})\right)$ and the
triple  $\left(K_0(\mathcal{A}),  K_0^+(\mathcal{A}), \Sigma(\mathcal{A})\right)$
are invariants of the Morita equivalence and isomorphism class of the 
$C^*$-algebra $\mathcal{A}$, respectively.

\subsection{Higher dimensional non-commutative tori}
\subsubsection{Definition and properties}
The even-dimensional   non-commutative torus $\mathcal{A}_{\Theta}$ is the universal $C^*$-algebra
generated by the unitary operators $U_1,\dots, U_{2n}$
and satisfying the  relations  $\{U_jU_i=e^{2\pi i\theta_{ij}} U_iU_j ~|~ 1\le i,j\le 2n\}$,
where $\Theta=(\theta_{ij})$ is  a skew-symmetric real matrix (\ref{eq1.2}). 
 Let  $SO(n, n;  ~\mathbf{Z})$ be a subgroup 
 of the group $GL(2n, \mathbf{Z})$ preserving the quadratic form $x_1x_{n+1}+x_2x_{n+2}+\dots+x_nx_{2n}$. 
The matrix  $\left(\small\begin{matrix} A & B\cr C & D\end{matrix}\right)\in SO(n, n;  ~\mathbf{Z})$,
if and only if,   
\begin{equation}\label{eq2.5}
A^TD+C^TB=I,\quad A^TC+C^TA=0=B^TD+D^TB,
\end{equation}
where  $A,B,C,D\in GL (n, \mathbf{Z})$ and $I$ is the identity matrix. 
The noncommutative tori  $\mathcal{A}_{\Theta}$ and $\mathcal{A}_{\Theta'}$ are
Morita equivalent, whenever 
\begin{equation}\label{eq2.6}
\Theta'={A\Theta+B\over C\Theta+D},  \quad\hbox{where}\quad  
\left(
\begin{matrix}
A & B\cr C & D
\end{matrix}
\right)\in SO(n, n;  ~\mathbf{Z}). 
\end{equation}

\medskip
The reader can verify, that  conditions  (\ref{eq2.5})
follow from the matrix equation (\ref{eq2.1}),  but not vice versa. 
In other words, one gets an inclusion $Sp ~(2n, \mathbf{Z})\subseteq SO(n, n;  ~\mathbf{Z})$.
This observation prompts the construction of a functor $F$ from the $n$-dimensional 
abelian varieties $\mathscr{A}$ to the $2n$-dimensional non-commutative 
tori $\mathcal{A}_{\Theta}$,  such that if $\mathscr{A}$ and $\mathscr{A}'$ 
are isomorphic abelian varieties,  the
$\mathcal{A}_{\Theta}=F(\mathscr{A})$ and $\mathcal{A}_{\Theta}'=F(\mathscr{A}')$
will be the Morita equivalent non-commutative tori, see  (\ref{eq2.2}) and (\ref{eq2.6}). 
The restriction of $F$ to the simple abelian varieties $\mathscr{A}_K$ corresponds to the 
non-commutative tori $\mathcal{A}_{\Theta(k)}=F(\mathscr{A}_K)$, where 
$\Theta(k)$ is the matrix over a number field $k$ of  $\deg~(k|\mathbf{Q})=2n$.

\subsubsection{Crossed product $\mathcal{A}_{\Theta (k)}\rtimes_{L_v}\mathbf{Z}$}
Let $\mathcal{A}_{\Theta}$ be a $2n$-dimensional  non-commutative torus
endowed with the canonical trace $\tau:  \mathcal{A}_{\Theta}\to\mathbf{C}$. 
Since $K_0(\mathcal{A}_{\Theta})\cong \mathbf{Z}^{2^{2n-1}}$
and $\tau_*: K_0(\mathcal{A}_{\Theta})\to\mathbf{R}$ is a homomorphism
induced by $\tau$, one gets a $\mathbf{Z}$-module  $\Lambda:=\tau_*(K_0(\mathcal{A}_{\Theta}))\subset\mathbf{R}$
of rank $2^{2n-1}$.  The generators of $\Lambda$ belong to the ring  $\mathbf{Z}[\theta_{ij}]$. 
In what follows, we assume that $\theta_{ij}\in k$.  In this case one gets the algebraic constraints and 
the rank of $\Lambda$ is equal to $2n$
 \cite[Remark 6.6.1]{N}. In other words,  
\begin{equation}\label{eq2.6.5}
 \Lambda\cong \mathbf{Z}+\mathbf{Z}\theta_1+\dots+\mathbf{Z}\theta_{2n-1}, \quad\hbox{where}  ~\theta_i\in k.
 \end{equation}
\begin{definition}\label{dfn2.1}
We shall write $B\in GL(2n, \mathbf{Z})$ to denote a positive matrix,
such that  $(1,\theta_1,\dots,\theta_{2n-1})$ is the normalized 
Perron-Frobenius eigenvector of $B$. 
\end{definition}
Let $\pi(n)$  be an integer-valued function defined in \cite[Section 6.5.3]{N}. 
Consider the characteristic polynomial $Char~\left(B^{\pi(p)}\right)=x^{2n}-a_1x^{2n-1}-\dots-a_{2n-1}x-1$
and a matrix 
\begin{equation}\label{eq2.7}
L_v=
\left(
\begin{matrix}
a_1     &  1    & \dots  & 0 &  0\cr
a_2     & 0      & 1  &  0      & 0\cr
\vdots  & \vdots  & \ddots & \vdots  & \vdots\cr
a_{2n-1} & 0  & \dots &  0 & 1 \cr
p      &  0     & \dots  &    0   & 0
\end{matrix}\right),
\end{equation}
where $p$ is the prime underlying $v$.   The  $L_v$ defines an 
endomorphism of the algebra  $\mathcal{A}_{\Theta(k)}$ by its action on the 
generators $U_1,\dots, U_{2n}$.  We consider the crossed product $C^*$-algebra 
$\mathcal{A}_{\Theta (k)}\rtimes_{L_v}\mathbf{Z}$ associated to such an action. 
\begin{remark}\label{rmk2.1}
There exits an isomorphism:
\begin{equation}\label{eq2.8}
K_0\left(\mathcal{A}_{\Theta (k)}\rtimes_{L_v}\mathbf{Z}\right)\cong \mathscr{A}_{\mathbf{F}_p},
\end{equation}
where  $\mathscr{A}_{\mathbf{F}_p}$ is a  localization of 
 the  $\mathscr{A}_K$ at the prime ideal $\mathscr{P}\subset K$ over $p$
  \cite[Section 6.6]{N}.   Formula (\ref{eq2.8}) hints that    the crossed product 
  $\mathcal{A}_{\Theta (k)}\rtimes_{L_v}\mathbf{Z}$
is  a  recast  of the variety $\mathscr{A}_{K_v}$. 
\end{remark}

\section{Proof of theorem \ref{thm1.1}}
We refer the reader to Section 1 for an outline of the proof. The detailed argument
is given below.  We split the proof in a series of lemmas. 
\begin{lemma}\label{lm3.1}
The scale  $\Sigma\left(\mathcal{A}_{\Theta (k)}\right)$ has the structure of a torsion group, so that 
$WC(\mathscr{A}_K)\cong \Sigma\left(\mathcal{A}_{\Theta (k)}\right)$,
where $\mathcal{A}_{\Theta(k)}=F(\mathscr{A}_K)$.
\end{lemma}
\begin{proof}
(i) For the sake of clarity, we restrict to the case $n=1$, i.e. when the variety 
$\mathscr{A}$ is an elliptic curve $\mathscr{E}$. In this case  matrix (\ref{eq1.2}) 
can be written as: 
\begin{equation}
\Theta=\left(\begin{matrix} 0 &\theta\cr -\theta & 0\end{matrix}\right). 
\end{equation}
We shall denote the corresponding non-commutative torus by $\mathcal{A}_{\theta}$. 
One gets 
\linebreak
 $\deg~(k|\mathbf{Q})=2$, i.e.  $k$ is a real quadratic  field.  
Therefore $F(\mathscr{E}_K)=\mathcal{A}_{\theta}$,  where  $\theta\in k$ is a quadratic irrationality. 

\medskip
To prove that the scale $\Sigma(\mathcal{A}_{\theta})$ has the structure of a torsion group,
we shall use  the Minkowski  question-mark function $?(x)$ [Minkowski 1904] \cite[p. 172]{Min1}.
Such a function is known to map quadratic irrational numbers of the unit interval to the rational 
numbers of the same interval preserving their natural order. This observation implies
 $\Sigma(\mathcal{A}_{\theta})\subset \mathbf{Q}/\mathbf{Z}$, i.e. the scale 
 $\Sigma(\mathcal{A}_{\theta})$ is a subgroup of the torsion group $\mathbf{Q}/\mathbf{Z}$.
Let us pass to a detailed argument. 
  
\medskip
 Let $\tau$ be the canonical trace on the algebra $\mathcal{A}_{\theta}$. Since $K_0(\mathcal{A}_{\theta})\cong \mathbf{Z}^2$,
 one gets $\tau_*(K_0(\mathcal{A}_{\theta}))=\mathbf{Z}+\mathbf{Z}\theta$ and  
 $\tau_*(K_0^+(\mathcal{A}_{\theta}))=\{m+n\theta\ge 0 ~|~m,n\in\mathbf{Z}\}$.
 It is known,  that  $\tau_*(u)=1$, where   $u\in  K_0^+(\mathcal{A}_{\theta})$ is an order unit. Therefore
 the  traces on  the scale  $\Sigma(\mathcal{A}_{\theta})$ are  given by the formula:
 \begin{equation}\label{eq3.2}
 \tau_*(\Sigma(\mathcal{A}_{\theta}))=[0,1]\cap\mathbf{Z}+\mathbf{Z}\theta. 
 \end{equation}

\medskip
Recall that the Minkowski question-mark function is defined by the convergent
series
 \begin{equation}
 ?(x):=a_0+2\sum_{k=1}^{\infty} {(-1)^{k+1}\over 2^{a_1+\dots+a_k}},
 \end{equation}
where $x=[a_0,a_1,a_2,\dots]$ is the  continued fraction of the irrational number $x$. 
The $?(x): [0,1]\to [0,1]$ is a monotone continuous function with the following properties:
(i) $?(0)=0$ and $?(1)=1$; (ii)  $?(\mathbf{Q})=\mathbf{Z}[{1\over 2}]$ are dyadic rationals
and  (iii) $?(\mathscr{Q})=\mathbf{Q}-\mathbf{Z}[{1\over 2}]$, where $\mathscr{Q}$ are quadratic 
irrational numbers  [Minkowski 1904] \cite[p. 172]{Min1}.

\medskip
Consider  subset (\ref{eq3.2}) of the interval $[0,1]$. Since $\theta$ is a quadratic irrationality,
we conclude that the set  $\tau_*(\Sigma(\mathcal{A}_{\theta}))\subset\mathscr{Q}$. 
By property (iii) of the Minkowski question-mark function, one gets an inclusion
$?(\tau_*(\Sigma(\mathcal{A}_{\theta})))\subset \mathbf{Q}/\mathbf{Z}$. Thus 
we constructed a one-to-one map:
\begin{equation}\label{eq3.4}
\Sigma(\mathcal{A}_{\theta})\to\mathscr{Y}\subset \mathbf{Q}/\mathbf{Z},
\end{equation}
where $\mathscr{Y}:=?(\tau_*(\Sigma(\mathcal{A}_{\theta})))$.  It follows  from (\ref{eq3.4}), 
that   $\Sigma(\mathcal{A}_{\theta})$ is  a torsion group. 
\begin{remark}
Formula (\ref{eq3.4}) is part of a duality between the K-theory of 
non-commutative tori and the Galois cohomology of abelian varieties. 
Such a study is beyond the scope of present paper. 
\end{remark}

\medskip
(ii) Let us show,  that  $WC(\mathscr{E}_K)\cong \Sigma(\mathcal{A}_{\theta})$,
where $\cong$ is an isomorphism of the torsion groups. 
Indeed, let $\mathscr{C}$ be the principal homogeneous space of the 
elliptic curve $\mathscr{E}_K$.  It is known, that $\mathscr{C}\cong \mathscr{E}_K'$
is the twist of $\mathscr{E}_K$, i.e. an isomorphic but not $K$-isomorphic elliptic 
curve $\mathscr{E}_K'$.  It follows,  that the  $\mathcal{A}_{\theta}$ and 
$\mathcal{A}_{\theta}'=F(\mathscr{E}_K')$ are  Morita equivalent, but not isomorphic
$C^*$-algebras.  
 This would  imply $\left(K_0(\mathcal{A}_{\theta}),  K_0^+(\mathcal{A}_{\theta})\right)\cong
 \left(K_0(\mathcal{A}_{\theta}'),  K_0^+(\mathcal{A}_{\theta}')\right)$,
 but  $\left(K_0(\mathcal{A}_{\theta}),  K_0^+(\mathcal{A}_{\theta}), \Sigma(\mathcal{A}_{\theta})\right)
 \not\cong\left(K_0(\mathcal{A}_{\theta}'),  K_0^+(\mathcal{A}_{\theta}'), \Sigma(\mathcal{A}_{\theta}')\right)$. 
 In other words, the principal homogeneous spaces of elliptic curve  $\mathscr{E}_K$ are classified 
 by the scales  of the algebra $\mathcal{A}_{\theta}$. 

\medskip
On the other hand, each element of $K_0^+(\mathcal{A}_{\theta})$
can be taken for an order unit $u$ of the dimension group 
 $\left(K_0(\mathcal{A}_{\theta}),  K_0^+(\mathcal{A}_{\theta})\right)$. 
 Since $\Sigma(\mathcal{A}_{\theta})=\{a\in K_0^+(\mathcal{A}_{\theta}) ~|~0\le a\le u\}$,
 we conclude that the elements of  $K_0^+(\mathcal{A}_{\theta})$ classify all 
 scales of the algebra $\mathcal{A}_{\theta}$. We can always restrict to the 
 generating subset $\Sigma(\mathcal{A}_{\theta})\subset K_0^+(\mathcal{A}_{\theta})$,
 since all other elements of the positive cone $K_0^+(\mathcal{A}_{\theta})$ correspond 
 to the finite unions $\mathscr{C}_1\cup\dots\cup\mathscr{C}_k$
 of the generating homogeneous spaces $\mathscr{C}_i$.  It remains to notice, that 
 equivalence classes of the principal homogeneous spaces $(\mathscr{C},\mu)$ 
 coincide with the isomorphism classes of the algebra $\mathcal{A}_{\theta}$. 
 In other words, one gets an isomorphism of the torsion groups:
\begin{equation}\label{eq3.5}
WC(\mathscr{E}_K)\cong \Sigma(\mathcal{A}_{\theta}). 
\end{equation}

\medskip
(iii) The general case $n>1$  is proved by an adaption of the argument for the case $n=1$. 
Notice that  one must use the Perron-Frobenius $n$-dimensional 
continued fractions and a higher-dimensional analog of the 
 Minkowski question-mark function. The details are left to the reader.

\medskip
Lemma \ref{lm3.1} is proved. 
\end{proof}

\begin{lemma}\label{lm3.2}
$WC(\mathscr{A}_{K_v})\cong 
K_0\left(\mathcal{A}_{\Theta (k)}\rtimes_{L_v}\mathbf{Z}\right)$,
where $\mathcal{A}_{\Theta(k)}=F(\mathscr{A}_K)$.
\end{lemma}
\begin{proof}
The abelian group $WC(\mathscr{A}_{K_v})$ is known to be finite,
see e.g. [Borel \& Serre 1964]  \cite{BorSer1}.  Specifically,
\begin{equation}\label{eq3.6}
WC(\mathscr{A}_{K_v})\cong 
(\widehat{\mathscr{A}}_{K_v})^*, 
\end{equation}
where $\widehat{\mathscr{A}}_{K_v}$  are the rational points of $K_v$ on the dual abelian
variety of $\mathscr{A}_{K_v}$ and $(\widehat{\mathscr{A}}_{K_v})^*$
is the character group of $\widehat{\mathscr{A}}_{K_v}$ [Tate 1958] \cite{Tat1}. 

\medskip
Recall that  $K_0\left(\mathcal{A}_{\Theta (k)}\rtimes_{L_v}\mathbf{Z}\right)\cong\mathscr{A}(\mathbf{F}_p)$,
 where $\mathcal{A}_{\Theta(k)}=F(\mathscr{A}_K)$ and  $\mathscr{A}(\mathbf{F}_p)$ is a localization of 
 the abelian variety $\mathscr{A}_K$ at the prime ideal $\mathscr{P}$ over prime $p$,
 see remark \ref{rmk2.1}. 
In view of the Hensel  Lemma, the points of $\mathscr{A}(\mathbf{F}_p)$ are 
the rational points of $\mathscr{A}_{K_v}$, where $v$ is the place over prime $p$. 
Since the torsion points of $\mathscr{A}_{K_v}$ and the torsion points of 
$\widehat{\mathscr{A}}_{K_v}$ coincide, we conclude that 
$\widehat{\mathscr{A}}_{K_v}\cong \mathscr{A}(\mathbf{F}_p)$.  
On the other hand, since finite abelian groups are Pontryagin self-dual,
one gets $(\widehat{\mathscr{A}}_{K_v})^*\cong  \mathscr{A}(\mathbf{F}_p)$. 
Therefore isomorphism (\ref{eq3.6}) can be written as:
\begin{equation}\label{eq3.7}
WC(\mathscr{A}_{K_v})\cong 
K_0\left(\mathcal{A}_{\Theta (k)}\rtimes_{L_v}\mathbf{Z}\right).  
\end{equation}
Lemma \ref{lm3.2} is proved. 
\end{proof}

\begin{corollary}\label{cor3.2}
$\prod_v 
WC(\mathscr{A}_{K_v})\cong \prod_v 
K_0\left(\mathcal{A}_{\Theta (k)}\rtimes_{L_v}\mathbf{Z}\right)$.
\end{corollary}
\begin{proof}
One takes the  direct sum of finite groups $WC(\mathscr{A}_{K_v})$
over all non-archimedean places $v$ of the field $K$.  
The corollary follows from isomorphism (\ref{eq3.7}).  
\end{proof}

\begin{lemma}\label{lm3.3}
$$\Sha (\mathscr{A}_K)\cong
\begin{cases}
Cl~(\Lambda)\oplus Cl~(\Lambda), & \mbox{if $k$ is even,}  \cr
\left(\mathbf{Z}/2^k\mathbf{Z}\right) \oplus Cl_{~\mathbf{odd}}~(\Lambda)\oplus  Cl_{~\mathbf{odd}}~(\Lambda),  & \mbox{if $k$ is odd.} 
\end{cases}$$
\end{lemma}
\begin{proof}
For the sake of clarity, let us outline the main idea. 
In view of lemma \ref{lm3.1} and corollary \ref{cor3.2}, 
one can replace (\ref{eq1.1}) by a group homomorphism:
\begin{equation}\label{eq3.8}
\Sigma\left(\mathcal{A}_{\Theta (k)}\right)\rightarrow  \prod_v 
K_0\left(\mathcal{A}_{\Theta (k)}\rtimes_{L_v}\mathbf{Z}\right). 
\end{equation}
To calculate the kernel $\Sha (\mathscr{A}_K)$ of (\ref{eq3.8}),  
recall that the algebra $\mathcal{A}_{\Theta (k)}$ at the LHS of (\ref{eq3.8})
depends on the similarity class of a single matrix $B\in GL (2n, \mathbf{Z})$,
see definition \ref{dfn2.1}.  On the other hand, the RHS of (\ref{eq3.8})
depends solely  on the characteristic polynomial of $B$. 
Roughly speaking, the problem of kernel of  (\ref{eq3.8}) 
reduces to an old question of the  linear algebra:
how many non-similar matrices  with  the same characteristic polynomial are there in
$GL (2n, \mathbf{Z})$? This problem was solved by  [Latimer \& MacDuffee 1933]
\cite{LatMac1}.  Their theorem says that the similarity classes of matrices 
$B\in GL (2n, \mathbf{Z})$ with the fixed polynomial $Char~(B)$ are in one-to-one 
correspondence with the ideal classes $Cl~(\Lambda)$ of an order $\Lambda$
of a number field $k$ generated by the eigenvalues of $B$. 
Coupled with the fact, that the scale $\Sigma\left(\mathcal{A}_{\Theta (k)}\right)$
contains the scales corresponding to all similarity classes, the  Latimer-MacDuffee Theorem
implies that the (half of) kernel of (\ref{eq3.8}) is  isomorphic  to   the group  $Cl~(\Lambda)$. 
The rest of the proof follows from the Atiyah pairing between the K-theory and  the K-homology. 
We pass to a step by step argument.

\bigskip
(i)  We  use notation of Sections 2.3  and  2.4.2. 
Recall that the positive cone $K_0^+(\mathcal{A}_{\Theta(k)})$  and 
the scale $\Sigma\left(\mathcal{A}_{\Theta (k)}\right)$ 
can be recovered  from (\ref{eq2.6.5})  by solving the inequality $\Lambda\ge 0$ and  
$0\le\Lambda\le 1$, respectively. 
Since $\theta_i\in k$ are components of the normalized Perron-Frobenius eigenvector
of a matrix $B\in GL(2n, \mathbf{Z})$, we conclude that the similarity class of $B$ 
defines the algebra $\mathcal{A}_{\Theta(k)}$ up to an isomorphism, and vice versa. 
The set of all pairwise non-similar matrices with the characteristic polynomial $Char ~(B)$ 
will be denoted by $B_1,\dots, B_h$.

\medskip
(ii) Let us show, that $\Lambda_1\subset\dots\subset\Lambda_h$,  where $\Lambda_i$ 
are $\mathbf{Z}$-modules corresponding to the matrices $B_i$. 
Indeed, since $B_i$ has the same characteristic polynomial as $B$, we conclude
that the components of the normalized Perron-Frobenius eigenvector of $B_i$
must lie in the number field $k$. Therefore $\Lambda_i$ is a full $\mathbf{Z}$-module
in the field $k$. It is well known, that set $\{\Lambda\}_{i=1}^h$ of such modules can be ordered  by 
inclusion  $\Lambda_1\subset\dots\subset\Lambda_h$, 
where $\Lambda_h$ is the maximal $\mathbf{Z}$-module. 
\begin{remark}\label{rmk3.6}
The inclusion of scales $\Sigma_1\subset\dots\subset\Sigma_h$ follows from the 
inclusion $\Lambda_1\subset\dots\subset\Lambda_h$ and the inequality
$0\le\Lambda_i\le 1$. 
\end{remark}

\medskip
(iii) Let us prove that the matrix $L_v$ given by formula (\ref{eq2.7}) 
does not depend on $\{B_i ~|~1\le i\le h\}$. 
Indeed,  the  $a_k$ in (\ref{eq2.7}) are coefficients
of the characteristic polynomial of the  matrix $B^{\pi (p)}$,  where 
$\pi (p)$ is a positive integer.  Notice that the spectrum 
$Spec ~(B_i)=\{\lambda_1^i,\dots,\lambda_{2n}^i\}$ does not depend 
on the matrix $B_i$, since $Char~(B_i)$ does not vary. 
Because  $Spec ~\left(B_i^{\pi (p)}\right)=\{(\lambda_1^i)^{\pi (p)},\dots,(\lambda_{2n}^i)^{\pi (p)}\}$,
we conclude that the spectrum of the matrix $B_i^{\pi (p)}$ is independent of $B_i$. 
Therefore the polynomial $Char~\left(B_i^{\pi (p)}\right)$ 
and  the coefficients $a_k$ in (\ref{eq2.7}) are the same for all matrices $B_i$. 
It follows from (\ref{eq2.7}), that  $L_v$ is independent of the choice of matrix $\{B_i ~|~1\le i\le h\}$.

\bigskip
(iv)  Let us calculate the kernel of homomorphism (\ref{eq3.8}).  Without loss of generality,
we assume $B\cong B_h$. By remark \ref{rmk3.6}, we have the inclusion of torsion groups:
\begin{equation}\label{eq3.9}
\Sigma_1\subset\dots\subset\Sigma_{h-1}\subset\Sigma\left(\mathcal{A}_{\Theta (k)}\right). 
\end{equation}
Let {\bf e} $\in \prod_v K_0\left(\mathcal{A}_{\Theta (k)}\rtimes_{L_v}\mathbf{Z}\right)$
be the trivial element of  torsion group at the RHS of (\ref{eq3.8}). 
From item (iii),  it is known that the preimage of {\bf e} under the homomorphism (\ref{eq3.8}) 
consists of $h$ distinct elements of the group   $\Sigma\left(\mathcal{A}_{\Theta (k)}\right)$. 
Indeed, each $\Sigma_i$ in (\ref{eq3.9}) has a unique such element lying in $\Sigma_i\backslash\Sigma_{i-1}$, 
since otherwise the corresponding abelian variety would have two different reductions modulo $p$. 
We conclude therefore, that  the kernel of homomorphism (\ref{eq3.8}) is an abelian group of order $h$. 

\bigskip
(v) Let us calculate the Shafarevich-Tate group  $\Sha (\mathscr{A}_K)$.  
It follows from item (iv), that $h=|Cl~(\Lambda)|$, where $Cl~(\Lambda)$
is the ideal class group of $\Lambda$, see   [Latimer \& MacDuffee 1933]
\cite{LatMac1}.  It is easy to see, that  the kernel of (\ref{eq3.8}) is isomorphic 
to the group $Cl~(\Lambda)$.  
To  express $\Sha (\mathscr{A}_K)$ in terms of  the group $Cl~(\Lambda)$, 
recall that the K-homology is  dual  to the K-theory  [Blackadar 1986] \cite[Section 16.3]{B}. 
Roughly speaking, the cocycles in $K$-theory are represented by the vector bundles.  Atiyah proposed using elliptic operators to
 represent the K-homology cycles.  An elliptic operator can be twisted by a vector bundle, and the Fredholm index of the 
 twisted operator defines a  pairing between the K-homology and the K-theory.
 In particular, $K^0(\mathcal{A}_{\Theta (k)})\cong K_0(\mathcal{A}_{\Theta (k)})$,  where 
$K^0(\mathcal{A}_{\Theta (k)})$ is the zero $K$-homology group of  the
the algebra $\mathcal{A}_{\Theta (k)}$.  
By analogy with  (\ref{eq3.4}), one gets
$\Sigma_0(\mathcal{A}_{\Theta (k)}),  \Sigma^0(\mathcal{A}_{\Theta (k)})\hookrightarrow\mathbf{Q}/\mathbf{Z}$
and, therefore, a pair of embeddings:
\begin{equation}\label{eq3.10}
Cl~(\Lambda_0),  ~Cl~(\Lambda^0)\hookrightarrow\mathbf{Q}/\mathbf{Z}. 
\end{equation}
Since $Cl~(\Lambda_0)\cong Cl~(\Lambda^0)$ are finite abelian groups,
the Atiyah pairing gives rise to a bilinear form $Q_{\mathbf{F}_p}(x,y)$ over the finite field $\mathbf{F}_p$
for each prime $p$ dividing $|Cl~(\Lambda)|$. 
The  (\ref{eq3.10}) is a group homomorphism, if and only if, the  $Q_{\mathbf{F}_p}(x,y)$
is  an alternating form, i.e. $\{Q_{\mathbf{F}_p}(x,x)=0 ~|~\forall x\in Cl~(\Lambda_0)\oplus Cl~(\Lambda^0)\}$.
Recall that the alternating bilinear forms  $Q_{\mathbf{F}_p}(x,y)$  exist, if and only if,  $p\ne 2$.  
Namely,  the form $Q_{\mathbf{F}_2}(x,y)$ is always symmetric, i.e. $Q_{\mathbf{F}_2}(x,x)\ne 0$ unless $x=0$. 
Thus there are no Atiyah pairing in the case $p=2$. The rest of the proof follows the Basis Theorem
for finite abelian groups.  Lemma \ref{lm3.3} is proved. 
\end{proof}

\begin{remark}
Both cases of lemma \ref{lm3.3} are realized by the concrete abelian varieties,
see  [Rubin 1989] \cite[Examples (B) and (C)]{Rub1} and [Poonen \& Stoll 1999] \cite[Proposition 27]{PooSto1},
respectively.
 \end{remark}

\bigskip
Theorem \ref{thm1.1} follows from lemma \ref{lm3.3}.

\section{Complex multiplication}
The  $\Sha (\mathscr{A}_K)$ can be calculated using  theorem \ref{thm1.1},
if the functor $F$ is  given explicitly.  To illustrate the idea, 
we consider the abelian varieties with complex multiplication.

\subsection{Abelian varieties}
Denote by  {\bf K}  a CM-field, i.e. the totally imaginary quadratic 
extension of a totally real number field {\bf k}.  
The abelian variety $\mathscr{A}_{CM}$ is said to have complex 
multiplication, if the endomorphism ring of $\mathscr{A}_{CM}$ contains 
the field {\bf K}, i.e. {\bf K}  $\subset End~\mathscr{A}_{CM}\otimes\mathbf{Q}$.  
The $deg ~(${\bf K} $|~\mathbf{Q})=2n$, where $n$ is the complex dimension of the 
$\mathscr{A}_{CM}$.  For a canonical basis in {\bf K} the lifting of the Frobenius 
endomorphism has the matrix form:
\begin{equation}\label{eq4.1}
Fr_v=
\left(
\begin{matrix}
a_1     &  1    & \dots  & 0 &  0\cr
-a_2     & 0      & 1  &  0      & 0\cr
\vdots  & \vdots  & \ddots & \vdots  & \vdots\cr
a_{2n-1} & 0  & \dots &  0 & 1 \cr
-p      &  0     & \dots  &    0   & 0
\end{matrix}\right).
\end{equation}
The functor $F$ is acting by the formula
$Fr_v\mapsto L_v$,  where $L_v$ is given by (\ref{eq2.7}),
see  \cite[p. 179]{N} for the details.  Since $L_v\in End~\Lambda$, one can recover 
the ring $\Lambda$ from the eigenvalues of matrix $L_v$. 
Let us consider  the simplest case $n=1$, i.e. the elliptic curves 
with complex multiplication.

\subsection{Elliptic curves}
Denote by $\mathscr{E}_{CM}$ an elliptic curve with complex multiplication 
by the ring $R=\mathbf{Z}+fO_{\mathbf{K}}$, where {\bf K} $=\mathbf{Q}(\sqrt{-D})$ 
is an imaginary quadratic field with the square free discriminant $D>1$  and $f\ge 1$
is the conductor of $R$.  The ring $\Lambda=F(R)$ is given by the formula
$\Lambda=\mathbf{Z}+f'O_k$, where $k=\mathbf{Q}(\sqrt{D})$ is the real quadratic 
field and the conductor $f'\ge 1$ is the least integer solution of the  equation
$|Cl~(\mathbf{Z}+f'O_k)|=|Cl~(R)|$, see \cite[Theorem 1.4.1]{N}. 
Moreover,  there exists a group isomorphism $Cl~(\Lambda)\cong Cl~(R)$.
Using  theorem \ref{thm1.1},  one gets the following corollary.
\begin{corollary}\label{cor4.1}
 $\Sha (\mathscr{E}_{CM})\cong  Cl~(R)\oplus Cl~(R)$. 
\end{corollary}

\medskip
\begin{example}
{\bf ([Rubin 1989] \cite[Example (B)]{Rub1}})
Let $\mathscr{E}_{CM}$ be the Fermat cubic $x^3+y^3=z^3$. 
The  $\mathscr{E}_{CM}$  has complex multiplication by the ring 
$R\cong \mathbf{Z}+O_{\mathbf{K}}$, where {\bf K} $\cong\mathbf{Q}(\sqrt{-3})$.
It is well known, that in this case $Cl~(R)$ is trivial.  We conclude from corollary 
\ref{cor4.1}, that the $\Sha (\mathscr{E}_{CM})$ is a trivial group. 
An alternative proof of this fact is based on the exact calculation of the $p$-part of 
 $\Sha (\mathscr{E}_{CM})$ and can be found in [Rubin 1989] \cite[Theorem 1 and Example (B)]{Rub1}. 
\end{example}
\begin{example}
{\bf ([Rubin 1989] \cite[Example (C)]{Rub1}})
Let $\mathscr{E}_{CM}$ be the modular curve $X_0(49)$ given by the equation
 $y^2+xy=x^3-x^2-2x-1$. 
The  $\mathscr{E}_{CM}$  has complex multiplication by the ring 
$R\cong \mathbf{Z}+O_{\mathbf{K}}$, where {\bf K} $\cong\mathbf{Q}(\sqrt{-7})$.
It is well known, that the group $Cl~(R)$ is trivial. Using corollary \ref{cor4.1},  one concludes 
that the $\Sha (\mathscr{E}_{CM})$ is a trivial group. 
A different proof of this result  can be found in
  [Rubin 1989] \cite[Theorem 1 and Example (C)]{Rub1}. 
\end{example}
\begin{remark}
Theorem \ref{thm1.1} is true for the simple abelian varieties $\mathscr{A}_K$. If  $\mathscr{A}_K$
 is not simple, then the functor $F$ splits and one gets  different formulas for the  group $\Sha (\mathscr{A}_K)$,
 see the corresponding calculations in the excellent paper  [Stein 2004] \cite[Theorem 3.1]{Ste1}.
\end{remark}

\bibliographystyle{amsalpha}


\end{document}